\documentclass[11pt]{amsart}
\usepackage{enumerate}
\usepackage{mathpazo}
\usepackage[scaled=.95]{helvet}
\usepackage{courier}
\advance \topmargin by -\headheight
\advance \topmargin by -\headsep
\evensidemargin \oddsidemargin
\theoremstyle{plain}
\newtheorem{thm}{Theorem}[section]
\newtheorem{prop}[thm]{Proposition}

\newtheorem{cor}[thm]{Corollary}
\theoremstyle{remark}
\newtheorem*{remark}{Remark}

\title{Pseudofree Group Actions on Spheres}
\author{Allan L. Edmonds}
\address{Department of Mathematics, Indiana University, Bloomington, IN 47405}
\email{edmonds@indiana.edu}

\begin{document}
\begin{abstract}
R.~S.~Kulkarni showed that a finite group acting pseudofreely, but not freely, preserving orientation, on an even-dimensional sphere (or suitable sphere-like space) is either a periodic group acting semifreely with two fixed points, a dihedral group acting with three singular orbits, or one of the polyhedral groups, occurring only in dimension 2.  It is shown here that the dihedral group does not act pseudofreely and locally linearly on an actual $n$-sphere when $n\equiv 0\mod 4$. The possibility of such an action when $n\equiv 2\mod 4$ and $n>2$ remains open. Orientation-reversing actions are also considered.
\end{abstract}
\dedicatory{Dedicated to Jos\'e Mar\'ia Montesinos on the occasion of his 65th birthday}
\maketitle

\section{Introduction}

The focus of this note is the question of what finite groups can act pseudofreely (but not freely) on some sphere. Recall that a group action is pseudofree if the fixed point set of each non-identity element is discrete.

A good part of this question was already answered by R.~S.~Kulkarni for pseudofree actions on (cohomology) manifolds with the homology of a sphere. But there were left open questions about existence of certain actions on actual spheres.  We quote from Kulkarni. (For consistency with the rest of this paper, we have made minor alterations in the notation.)

\begin{thm}[Kulkarni  \cite{Kulkarni1982}]
Let $X$ be an admissible space which is a $d$-dimensional $\mathbb{Z}$-cohomology manifold with the mod $2$ cohomology isomorphic to that of an even dimensional sphere $S^{n}$. Let $G$ be a finite group acting pseudofreely on $X$ and trivially on $H^{*}(X;\mathbb{Q})$. The either
\begin{enumerate}[a.]
\item $G$ acts semifreely with two fixed points and has periodic cohomology of period $d$, or
\item $G\approx $ a dihedral group of order $2k$, $k$ odd or
\item $n=2$ and $G\approx $ a dihedral, tetrahedral, octahedral or icosahedral group.
\end{enumerate}
\end{thm}

Actions of the first type in the theorem arise as suspensions of free actions. Actions of the third type arise as classical actions on the $2$-sphere.
Kulkarni, however, remarked (p. 222) that he did not know whether a dihedral group of order $2k$, $k$ odd, actually can act pseudofreely on a $\mathbb{Z}$-cohomology manifold which is a $\mathbb{Z}_{2}$-cohomology sphere of even dimension $>2$.

It should be remarked that aside from the standard actions on spheres $S^{1}$ and $S^{2}$ there are no such pseudofree \emph{linear} actions.

We restrict attention primarily to locally linear actions on actual closed manifolds,
and we find it useful to consider the slightly broader class of \emph{tame} actions. We understand a pseudofree action of a finite group $G$ to be \emph{tame} if each point $x$ has a disk neighborhood invariant under the isotropy group $G_{x}$.

We denote the dihedral group of order $2k$ by $D_{k}$ and the cyclic group of order $k$ by $C_{k}$. We fix an expression of $D_{k}$ as a semidirect product with group extension
\[
1\to C_{k} \to D_{k} \to C_{2} \to 1
\]
where the quotient $C_{2}$ acts on the normal subgroup $C_{k}$ by inversion. There are $k$ involutions in $D_{k}$, all conjugate, and all determining the various splittings as a semidirect product.

\begin{thm}\label{thm:nonexistence}
The dihedral group $D_{k}$ of order $2k$, $k$ odd, does not act locally linearly or tamely, pseudofreely, and preserving orientation, on $S^{n}$ when $n\equiv 0\mod 4$.
\end{thm}

In fact the argument shows that there do not exist orientation-preserving actions of $D_{k}$ on closed $n$-manifolds, $n\equiv 0\mod 4$, having exactly three orbit types $(2,2,k)$, i.e., with isotropy groups $C_{2}, C_{2}, C_{k}$.

We do not consider here the possibilities for nontrivial pseudofree actions on higher dimensional manifolds with the cohomology of $S^{n}$, such as $S^{n}\times\mathbb{R}^{m}$.

It remains to ponder the case $n\equiv 2\mod 4$, which encompasses the classical actions on the $2$-sphere. In this case we obtain a weak positive result.

\begin{thm}\label{thm:existence}
If $n\equiv 2\mod 4$ and $k$ is an odd positive integer, then there is a smooth, closed, orientable $n$--manifold on which the dihedral group $D_{k}$ of order $2k$ acts smoothly, pseudofreely, preserving orientation, with exactly three singular orbits of types $(2,2,k)$.
\end{thm}

The argument shows that when $n\ge 6$ such $n$-manifolds can be chosen to be $2$-connected. But it remains an open question whether the manifold can be chosen to be a sphere or a mod 2 homology sphere.

We also consider the case of orientation-reversing pseudofree actions on spheres.
\begin{thm}\label{thm:orientation_reversing_cases}
Suppose that a finite group $G$ acts locally linearly and pseudofreely on a sphere $S^{n}$, with some elements of $G$ reversing orientation. If $n$ is odd, then $G$ must be a dihedral group and $n\equiv 1\mod 4$; and if $n$ is even, then $G$ must be a periodic group with a subgroup of index $2$.\end{thm}

When $n$ is odd the prototype is the standard action of the dihedral group on a circle, but we do not know if there are analogous actions in higher odd dimensions. The existence is closely related to the existence of orientation-preserving actions in neighboring even dimensions.

When $n$ is even, these kinds of actions arise as  ``twisted suspensions'' of free actions.  Such actions by even order cyclic groups have been studied and classified by S.~E.~Cappell and J.~L.~Shaneson \cite{CappellShaneson1978}, in the piecewise linear case, and by S.~Kwasik and R.~Schultz \cite{KwasikSchultz1990, KwasikSchultz1991}, in the purely topological case.

\section{Proof of Theorem \ref{thm:nonexistence}}
Suppose that $D_{k}$ acts pseudofreely on $S^{n}$. Such an action cannot be free since $D_{k}$ does not satisfy the Milnor condition that every element of order two lies in the center. Similarly $n$ cannot be odd. For otherwise a nontrivial isotropy group would act freely, preserving orientation, on an even-dimensional sphere linking a point with nontrivial isotropy group. But this would violate the Lefschetz fixed point theorem.

So henceforth we assume that $n>2$ and $n$ is even. Now by the Lefschetz Fixed Point theorem we also conclude that every nontrivial element of $D_{k}$ has exactly two fixed points. The two fixed points of an element of order $k$ are interchanged by all the elements of order $2$, since otherwise the dihedral group would act freely on a linking sphere to one of the fixed points. On the other hand the cyclic subgroup $C_{k}$ permutes the fixed points of the elements of order two in two orbits of size $k$. Now removing small invariant disk neighborhoods of the singular set and passing to the quotient we obtain an $n$-manifold $Y^{n}$ whose boundary consists of two homotopy real projective $(n-1)$-spaces $P_{1}$ and $P_{2}$ and a single homotopy lens space $L_{k}$. We have $\pi_{1}(Y)=D_{k}$, $\pi_{1}(P_{j})\approx C_{2}$, and $\pi_{1}(L_{k})=C_{k}$.

The regular covering over $Y$ is classified by a map $f:Y\to K(D_{k},1)$. Note that although there are $k$ different subgroups of order $2$, they are all conjugate in $D_{k}$ and hence there is a well-defined inclusion-induced homomorphism $H_{*}(C_{2})\to H_{*}(D_{k})$, as well as the usual homomorphism $H_{*}(C_{k})\to H_{*}(D_{k})$.

We will make use of the following elementary, well-known,  homology calculation. The proof is an exercise in the spectral sequence of the split extension
$
1\to C_{k}\to D_{k}\to C_{2}\to 1.
$

\begin{prop}\label{prop:trivialcoefs}
For $k$ an odd integer, the homology of $D_{k}$ with $\mathbb{Z}$ coefficients is given by
\[
H_{q}(D_{k};\mathbb{Z})=
\begin{cases}
\mathbb{Z}  &  \text{for } q=0\\
\mathbb{Z}/2 &  \text{for } q\equiv 1\mod 4\\
\mathbb{Z}/2k &  \text{for } q\equiv 3\mod 4\\
0 &  \text{for even } q>0
\end{cases}
\]
Moreover, when $ q\equiv 3\mod 4$,  the inclusion $C_{k}\to D_{k}$  induces an injection $$ \mathbb{Z}_{k}=H_{q}(C_{k};\mathbb{Z})\to H_{q}(D_{k};\mathbb{Z})=\mathbb{Z}_{2k}$$ and the projection $D_{k}\to C_{2}$ induces a surjection $$ \mathbb{Z}_{2k}=H_{q}(D_{k};\mathbb{Z})\to H_{q}(C_{2};\mathbb{Z})=\mathbb{Z}_{2}.$$ \qed
\end{prop}

\begin{proof}[Completion of Proof of Theorem \ref{thm:nonexistence}]
The proof when $n\equiv 0\mod 4$ follows easily from Proposition \ref{prop:trivialcoefs}. Under the classifying map  $f:Y\to K(D_{k},1)$ restricted to $\partial Y$, $f_{*}[P_{i}]$ is the element of order $2$ in $H_{n-1}(D_{k})=\mathbb{Z}/2k$, for $i=1,2$. The element $f_{*}[L_{k}]$ is an element of order $k$. But obviously $f_{*}[P_{1}]+f_{*}[P_{2}]+f_{*}[L_{k}]=0$, since the classifying map is defined on all of $Y$. This implies $f_{*}[L_{k}]=0$, a contradiction.
\end{proof}

\begin{remark}
This argument shows that when $n\equiv 0\mod 4$ there is no orientation-preserving, pseudofree action of $D_{k}$ on any closed, orientable, $n$-manifold of type $(2,2,k)$, i.e., having singular orbit structure consisting of one orbit with isotropy group $C_{k}$ and two orbits with isotropy groups of order $2$.
\end{remark}

\section{Proof of Theorem \ref{thm:existence}}
We will also need a somewhat less precise statement for oriented bordism.
\begin{prop}\label{prop:}
For $k$ odd and $q\equiv 1\mod 4$, the map $\Omega_{q}(BC_{k})\to \Omega_{q}(BD_{k})$ induced by inclusion is zero.
\end{prop}
\begin{proof}
We indicate a proof, based on ``big guns''. According to results of Thom and Milnor, the oriented cobordism ring $\Omega_{*}$, is finitely generated, has no odd order torsion, and is finite except in dimensions divisible by $4$. Indeed all 2-torsion elements have order exactly 2, and the torsion subgroup is finitely generated in each dimension. Moreover, modulo torsion,  $\Omega_{*}$ is a polynomial algebra with one generator in each dimension $\equiv 0\mod 4$. Rationally these generators can be take to be complex projective spaces $\mathbb{C}P^{2m}$. For these and related facts, we refer to  R.~E.~Stong \cite{Stong1968}.

Now we can calculate both $\Omega_{q}(BC_{k})$ and $\Omega_{q}(BD_{k})$ via the Atiyah-Hirzebruch spectral sequences
\[
H_{i}(BC_{k};\Omega_{j})\Rightarrow \Omega_{i+j}(BC_{k})
\]
and
\[
H_{i}(BD_{k};\Omega_{j})\Rightarrow \Omega_{i+j}(BD_{k})
\]
Since is $k$  is odd, the groups $H_{i}(BC_{k};\Omega_{j})$ are zero unless $j\equiv 0\mod 4$, according to the above remarks. Therefore suppose $j\equiv 0\mod 4$.

Now $q=i+j$ is congruent to $1\mod 4$ if and only if $i\equiv 1\mod 4$. But  then 
$H_{i}(BC_{k};\Omega_{j})$ is $k$--torsion, since $H_{i}(BC_{k};\mathbb{Z})$ is $k$-torsion (for $i\ne 0$), while $H_{i}(BD_{k};\Omega_{j})$ is $2$--torsion by the Universal Coefficient formula. Therefore, in all such cases, with $i+j\equiv 1\mod 4$, the map
\[
H_{i}(BC_{k};\Omega_{j})\to
H_{i}(BD_{k};\Omega_{j})
\]
between the spectral sequences is trivial. The result follows by comparison of spectral sequences.
\end{proof}

\begin{cor}
If $L$ is a homotopy lens space of dimension $q\equiv 1\mod 4$ and fundamental group $\pi_{1}(L)=C_{k}$, then the composition 
$L\to BC_{k}\to BD_{k}$ of the classifying map with the inclusion-induced map is null-bordant.\qed
\end{cor}

\begin{proof}[Proof of Theorem \ref{thm:existence}]
We show how to construct a smooth, pseudofree action of $D_{k}$ on a smooth $n$-manifold, for any $n\equiv 2\mod 4$, with this same orbit structure. At this writing we are not sure whether the manifold can be chosen to be a sphere or even a mod 2 homology sphere, however, when $n>2$. We conjecture that it cannot.

In dimension $2$ there is a standard $D_{k}$ action on the $2$--sphere such that when one removes invariant disk neighborhoods of the singular points and passes to the orbit space one has a disk with two holes, i.e., a pair of pants. One boundary circle has isotropy type $C_{k}$ and the other two boundary circles have isotropy type $C_{2}$. One can think of the classifying map we examined above as given by choosing two distinct elements of order $2$ in $D_{k}$, with product of order $k$. 

We now consider higher dimensions $n\equiv 2\mod 4$. 
Start with a disjoint union of two real projective $(n-1)$-spaces $P_{1}$ and $P_{2}$ and a single lens space $L_{k}$. We define a regular $D_{k}$ covering of $P_{1}\sqcup  P_{2}\sqcup  L_{k}$ by mapping $P_{i}\to K(D_{k},1)$, representing the nonzero element of $H_{n-1}(D_{k})=\mathbb{Z}/2$, that is taking the canonical map $P_{i}\to K(C_{2},1)$ followed by a map $K(C_{2},1)\to K(D_{k},1)$ induced by an inclusion $C_{2}\to D_{k}$. Similarly we take a standard inclusion $L_{k}\to K(C_{k},1)$ composed with the natural map $K(C_{k},1)\to K(D_{k},1)$. Note that $\mathbb{R}P^{n-1}$ admits an orientation-reversing diffeomorphism, hence represents an element of order $2$ in $\Omega_{n-1}(BC_{2})$.

It follows from the preceding remarks that the combined map $P_{1}\sqcup  P_{2}\sqcup  L_{k}\to K(D_{k},1)$ is null-bordant. Indeed, $P_{1}\sqcup  P_{2}\to K(D_{k},1)$ and $L_{k}\to K(D_{k},1)$ are separately null-bordant.

Choose such a manifold with the desired boundary and a $D_{k}$ covering extending the given one.  Passing to the $2k$-fold covering and capping off all the boundary spheres with disks provides the required $n$-manifold with pseudofree action of $D_{k}$ with the desired singular orbit structure.
\end{proof}

\begin{remark}Of course we can arrange that the manifold constructed is connected, by forming the connected sum of components in the orbit space and noting that the resulting classifying map for the covering, $W^{n}\to BD_{k}$ must be surjective on fundamental group.
We can also easily arrange that the manifold with group action constructed above be simply connected. Just use surgery to kill the normal subgroup of the fundamental group of the oriented manfold with boundary $P_{1}\sqcup  P_{2}\sqcup  L_{k}$ with quotient group $D_{k}$.
One can further arrange that the manifold with group action is $2$-connected.
According to Stong \cite{Stong1968}, for example, 
$\Omega_{q}^{\text{Spin}}$, like $\Omega_{q}$, has no odd order torsion and has elements of infinite order only in dimensions divisible by $4$. Using this, the preceding spectral sequence argument shows that $\Omega_{q}^{\text{Spin}}(BC_{k})\to \Omega_{q}^{\text{Spin}}(BD_{k})$ is $0$ for $k$ odd
 and $q\equiv 1\mod 4$.  From this it follows that the $n$-manifold with group action can be chosen to be $2$-connected (when $n\ge 6$) by arranging that the orbit manifold is spin and then doing spin surgery on $0$-, $1$-  and $2$-spheres in the orbit space.
\end{remark}
\section{Proof of Theorem \ref{thm:orientation_reversing_cases}: The orientation-reversing cases}
Suppose a finite group $G$ acts pseudofreely on $S^{n}$, but with not every element of $G$ preserving orientation. Let $H<G$ be the subgroup of index two that does preserve orientation. 

\subsection{Dimension $n$ odd}
The fundamental example in this case is the action of the dihedral group on the unit circle.

It follows from the Lefschetz Fixed Point Formula that each orientation-reversing element has exactly two fixed points. The pseudofree condition implies that nontrivial, orientation-preserving elements act without fixed point. That is, the subgroup $H$ acts freely.
Thus each orientation-reversing element has order two, since the square of an orientation-reversing element has fixed points and the orientation-preserving subgroup acts freely.
Moreover, each orientation-reversing element $x\in G- H$ acts on $y\in H$ by inversion, for $xyxy=e\Rightarrow xyx=y^{-1}$. The fact that inversion is an automorphism of $H$ implies that $H$ is abelian. Since $H$ acts freely on a sphere it satisfies the property that every subgroup of order $p^{2}$ is cyclic. It follows that $H$ is cyclic of some order, hence that $G$ is dihedral.

It remains to decide whether one can actually construct such orientation-reversing dihedral actions in higher dimensions. If there were such an orientation-reversing action in an odd dimension $n$, then one could promote it to an orientation-preserving, ``tame'' pseudofree action in dimension $n+1$ by twisted suspension. It then follows from Kulkarni's result that $H$ has odd order. We would therefore conclude that $n+1\equiv 2\mod 4$, by Theorem \ref{thm:nonexistence}. Thus we have ruled out $n\equiv 3\mod 4$, and must have $n\equiv 1\mod 4$.

\subsection{Dimension $n$ even}
It follows from the Lefschetz Fixed Point Formula that no orientation-reversing (pseudofree) element  has a fixed point.

In this case, by the results from the preceding section, the orientation-preserving subgroup $H$ must be one of those described by Kulkarni.

\subsubsection{$H$ a periodic group acting semifreely with two fixed points} 
Deleting the two $H$ fixed points, we see that $G$ acts freely on $S^{n-1}\times\mathbb{R}$, and hence $G$ must have periodic cohomology and has $H$ as a subgroup of index 2, as required. \qed

\begin{remark}
A standard example arises when $G$ acts freely, preserving orientation, on an odd dimensional equatorial sphere. Such an action can be extended by twisted suspension, using a projection $G\to \{\pm 1\}$.
Cappell and Shaneson  \cite{CappellShaneson1978} argue that every PL pseudofree action of $\mathbb{Z}_{2N}$ is a twisted suspension. 
They also point out that not every such action is equivalent to a twisted suspension, for instance, for the quaternion group of order 8. Note also that an even order periodic group need not have a subgroup of index 2. An example is the binary icosahedral group, which is perfect.
\end{remark}

\subsubsection{$H$ a dihedral group $D_{k}$, acting with three singular orbits of types $2,2,k$}
Note that we have already ruled this case out when $n\equiv 0\mod 4$.  With the extra orientation-reversing elements,  we are able to rule out such actions in all cases  when $n\equiv 0\mod 2$, as we now explain.

Now a transfer argument shows that the orbifold $X=S^{n}/H$ has the rational homology of $S^{n}$, since $H$ acts homologically trivially. In particular, $\chi(S^{n}/H)=2$. On the other hand the same transfer argument shows that the orbifold $X=S^{n}/G$ has the rational homology of a point, since $G$ acts homologically nontrivially. In particular, $\chi(S^{n}/G)=1$.

It follows that the action of $G/H\approx C_{2}$ on $S^{n}/H$ has no fixed points.  On the other hand, the action of $G/H\approx C_{2}$ on $S^{n}/H$ must preserve the image of the \emph{three} $H$ singular orbits. And one of these singular orbits (at least the one of type $k$) must be fixed by $G/H$. This contradiction completes the proof. \qed

\bibliographystyle{amsplain}

\begin{thebibliography}{1}

\bibitem{CappellShaneson1978}
Sylvain~E. Cappell and Julius~L. Shaneson, \emph{Pseudofree actions. {I}},
  Algebraic topology, {A}arhus 1978 ({P}roc. {S}ympos., {U}niv. {A}arhus,
  {A}arhus, 1978), Lecture Notes in Math., vol. 763, Springer, Berlin, 1979,
  pp.~395--447. \MR{MR561231 (81d:57034)}

\bibitem{Kulkarni1982}
R.~S. Kulkarni, \emph{Pseudofree actions and {H}urwitz's {$84(g-1)$} theorem},
  Math. Ann. \textbf{261} (1982), no.~2, 209--226. \MR{MR675735 (84e:57034)}

\bibitem{KwasikSchultz1990}
S{\l}awomir Kwasik and Reinhard Schultz, \emph{Pseudofree group actions on
  {$S\sp 4$}}, Amer. J. Math. \textbf{112} (1990), no.~1, 47--70. \MR{MR1037602
  (90m:57034)}

\bibitem{KwasikSchultz1991}
\bysame, \emph{Topological pseudofree actions on spheres}, Math. Proc.
  Cambridge Philos. Soc. \textbf{109} (1991), no.~3, 433--449. \MR{MR1094744
  (92e:57053)}

\bibitem{Stong1968}
Robert~E. Stong, \emph{Notes on cobordism theory}, Mathematical notes,
  Princeton University Press, Princeton, N.J., 1968. \MR{MR0248858 (40 \#2108)}

\end{thebibliography}
\providecommand{\bysame}{\leavevmode\hbox to3em{\hrulefill}\thinspace}
\providecommand{\MR}{\relax\ifhmode\unskip\space\fi MR }
\providecommand{\MRhref}[2]{%
  \href{http://www.ams.org/mathscinet-getitem?mr=#1}{#2}
}
\providecommand{\href}[2]{#2}

\end{document}